\documentclass[12pt,psamsfonts]{amsart}
\usepackage{amsmath,amsthm,amsfonts,amssymb}
\usepackage{eucal,amscd}
\usepackage{graphicx}
\usepackage[all,knot]{xy}
\xyoption{arc}

\newcommand{\ot}{\otimes}
\newcommand{\mC}{\mathcal{C}}

\numberwithin{equation}{section}

\newtheorem{theorem}[equation]{Theorem}
\newtheorem{defn}[equation]{Definition}

\newtheorem{conj}[equation]{Conjecture}

\theoremstyle{definition}

\newtheorem{remark}[equation]{Remark}

\newtheorem{ex}[equation]{Example}
\newtheorem{definition}[equation]{Definition}

\begin{document}

\title[gYBE-RFC]
{Solutions to generalized Yang-Baxter equations via ribbon fusion categories}

\author{Alexei Kitaev}
\email{kitaev@iqi.caltech.edu}
\address{California Institute of Technology\\Pasadena, CA 91125\\USA}

\author{Zhenghan Wang}
\email{zhenghwa@microsoft.com}
\address{Microsoft Station Q\\CNSI Bldg Rm 2237\\
    University of California\\
    Santa Barbara, CA 93106-6105\\
    U.S.A.}

\dedicatory{Dedicated to Mike Freedman on the occasion of his $60^{th}$ birthday}

\thanks{The second author is partially supported by NSF DMS 1108736 and would like to thank E. Rowell for observing $(3)$ of Thm. \ref{MainThm}, S. Hong for helping on $6j$ symbols, and R. Chen for numerically testing the solutions.}

\begin{abstract}

Inspired by quantum information theory, we look for representations of the braid groups $B_n$ on $V^{\otimes (n+m-2)}$ for some fixed vector space $V$ such that each braid generator $\sigma_i, i=1,...,n-1,$ acts on $m$ consecutive tensor factors from $i$ through $i+m-1$. The braid relation for $m=2$ is essentially the Yang-Baxter equation, and the cases for $m>2$ are called generalized Yang-Baxter equations.  We observe that certain objects in ribbon fusion categories naturally give rise to such representations for the case $m=3$.  Examples are given from the Ising theory (or the closely related $SU(2)_2$), $SO(N)_2$ for $N$ odd, and $SU(3)_3$.  The solution from the Jones-Kauffman theory at a $6^{th}$ root of unity, which is closely related to $SO(3)_2$ or $SU(2)_4$, is explicitly described in the end.

\end{abstract}

\maketitle

\section{Introduction}

A constant Yang-Baxter operator is an invertible linear operator $R: V\otimes V\longrightarrow V\otimes V$ which satisfies the braided version of the Yang-Baxter equation (YBE):
\vspace{.1in}
\begin{enumerate}\label{YBE:operatorform}
 \item[(YBE)] $(R\otimes I_V)(I_V\otimes R)(R\otimes I_V)=(I_V\otimes R)(R\otimes I_V)(I_V\otimes R)$.
\end{enumerate}
\vspace{.1in}
Any Yang-Baxter operator leads to a representation of the braid group $B_n$ by attaching $V$ to a braid strand, i.e., assigning $I^{\otimes (i-1)}\otimes R \otimes I^{\otimes (n-i-1)}$ to the braid generator $\sigma_i, i=1,\cdots, n-1$.  Inspired by quantum information theory, $R$ is recently interpreted as an entangling operator of $V\otimes V$.    Thinking along these lines, we could equally consider operators $R: V^{\otimes m}\longrightarrow V^{\otimes m}$ for any $m\geq 3$ that satisfy some generalized version of the YBE.  Such a generalization is proposed in \cite{RZWG}.  There are deep theories that lead to solutions of the YBE.  A natural question is how to find solutions to some generalized Yang-Baxter equations (gYBEs).  There are essentially four new solutions \cite{RZWG}\cite{RC} that are not derived from solutions to the YBE.  Though solutions to the YBE always lead to braid group representations for all $n$, this is not the case for general solutions to gYBEs because far commutativity is not automatically satisfied (See Section $4.1$ of \cite{RC}).  However, the four known solutions do satisfy far commutativity.  In this note, we point out a connection between the gYBE and the ribbon fusion category theory.  By finding a suitable object in a ribbon fusion category, we provide a solution to the gYBE for $m=3$ (or the $(d,3,1)$-gYBE in the terminology of \cite{RC}) which automatically leads to a braid group representation.

\section{Generalized Yang-Baxter-Equation object}

We follow the terminologies in Section $2$ of \cite{RC} for the gYBE and in Chapter $4$ of \cite{Wang} for ribbon fusion categories.  By far commutativity for the gYBE, we mean the condition in Definition $4.1$ of \cite{RC} as adapted in Definition \ref{gfarcommuntativity} below.

\begin{definition}
Let $V$ be a complex vector space of dimension $d$.  The $(d,m,1)$-gYBE is an equation for an invertible operator $R: V^{\ot m}\rightarrow V^{\ot m}$ such that
\vspace{.1in}
    \begin{enumerate}\label{gYBE:operatorform}
    \item[(gYBE)] $(R\otimes I_V) (I_V\ot R) (R\otimes I_V)=(I_V\ot R) (R\ot I_V) (I_V\ot R)$,
    \end{enumerate}
\vspace{.1in}
where $m$ is a natural number and $I_V$ is the identity operator on $V$.

Any matrix solution to the $(d,m,1)$-gYBE is called a $(d,m,1)$-$R$-matrix.

Note that the form of the $(d,m,1)$-gYBE is exactly the same as the YBE, thought the YBE is the $(d,2,1)$-gYBE in this language.
\end{definition}

If $R$ is a solution to a $(d,m,1)$-gYBE, for the braid generator $\sigma_i\in B_n$, set
$$R_{\sigma_i}=I^{\ot{(i-1)}} \otimes R \otimes I^{\ot{(n-i-1)}}.$$
Again this is exactly the same assignment as in the YBE case, but $B_n$ acts on the vector space $V^{\ot (n+m-2)}$.  The braid relation holds because $R$ satisfies the gYBE.
But far commutativity is not necessarily satisfied when $|i-j|>1$.

\begin{defn}\label{gfarcommuntativity}
Suppose $R$ is a $(d,m,1)$-$R$-matrix.  Then $R$ satisfies far commutativity if for each $j$ such that $2<j<m+1$,
$$R_{\sigma_1}R_{\sigma_j}=R_{\sigma_j}R_{\sigma_1}$$ in $\textrm{End}(V^{\ot (j-1+m)})$ for the braid generator $\sigma_j\in B_{j+1}$.
\end{defn}

If a $(d,m,1)$-$R$-matrix satisfies far commutativity, then it yields a representation of each $n$-strand braid group $B_n$ on $V^{\ot (n+m-2)}$.

There is a family of solutions to the $(2,3,1)$-gYBE that comes from the ribbon fusion category theory. The solutions and representations can be conveniently described by the tree bases of morphism spaces. There are constraints on the labels for the labeled trees to become a basis of general morphism spaces.  Therefore, the representation spaces are not necessarily tensor products. But in certain categories we can form invariant subspaces which have tensor product structures and are the desired braid representations. Some choices are formalized by the following definition.

\begin{definition}
An object $X$ in a ribbon fusion category $\mC$ is called a $(d,3,1)$-gYBE object with respect to a set of objects $S=\{X_i\}_{i\in I}$ of $\mC$ if the objects $\{X_i\}_{i\in I}$ are $d$ simple objects of distinct types of $\mC$ such that for each $i\in I$, $X\otimes X_i\cong \oplus_{j\in I}X_j$.
\end{definition}

\begin{ex}
Note that in the definition of a $(d,3,1)$-gYBE object, $X$ is not necessarily a simple object.  An example of a non-simple $(2,3,1)$-gYBE is $1\oplus \psi$ with respect to $\{1, \psi\}$ in the Ising theory.  Actually, the  $(2,3,1)$-gYBE object $1\oplus \psi$ leads to the following $(2,3,1)$-$R$-matrix:
$$
\begin{pmatrix}
1& 0 & 0 &0 &0 &0 &0 &0 \\
0& 0 &0 &1 &0 &0 &0 &0 \\
0 &0 &-1 &0 &0 &0 &0 &0 \\
0& 1& 0 & 0 &0 &0 &0 &0 \\
0&0&0&0& 0 &0 &1 &0 \\
0&0&0&0&0&-1&0& 0 \\
0&0&0&0& 1 &0 &0 &0 \\
0&0&0&0& 0 & 0&0 &1
\end{pmatrix}.$$
\end{ex}

Suppose $X$ is a gYBE object with respect to a set of objects $S=\{X_i\}_{i\in I}$.  Then there is a representation $V_{X_i, X^{\otimes n}, X_j}$ of the braid group $B_n$ obtained by considering the $(n+1)$-punctured disk $D^2_{X_i, X^{\otimes n}, X_j}$, where the $(n+1)$-punctures are labeled by one $X_i$ and $n$ $X$'s together with the boundary labeled by $X_j$.  In anyonic language, there are $(n+1)$ anyons---one $X_i$ and $n$ $X$'s---in the disk with total charge $X_j$.  A convenient orthonormal basis of $V_{X_i, X^{\otimes n}, X_j}$ is given, in graphical notation, by all admissible labelings of the following tree,  where $i_1,\cdots, i_{n-1}\in I$:

\[ \xy
(0,10)*{}="A"; (4,10)*{}="B";(8,10)*{}="C"; (14,10)*{}="K"; (14,-10.2)*{}="L";
(12,-10)*{}="D";(18,-16.2)*{}="H";(20,-19.1)*{}="M";(18,10)*{}="N";
(4,3.8)*{}="E";(8,-1.8)*{}="F";
"A";"M" **\dir{-}; "B";"E"**\dir{-};"C";"F" **\dir{-}; "K";"L" **\dir{-};"N";"H" **\dir{-};
(0,12)*{i};(4,12)*{X};(8,12)*{X};(20,-22)*{j}; (14,12)*{X}; (18,12)*{X};(11,6)*{...}; (5.5,-1.24)*{i_1};
(10,-6.6)*{\ddots}; (14,-15.3)*{i_{n-1}};
\endxy \]

Note that when $X$ is a $(d,3,1)$-gYBE object with respect to a set of objects $S=\{X_i\}_{i\in I}$, then every simple object $X_i, i\in I$ is an admissible label for any internal slant edge of the tree above.  Since all morphism spaces $\textrm{Hom}(X_i, X\otimes X_j)$ are $1$-dimensional,  there are no vertex labelings.

The representation of a braid $\sigma$ on $V_{X_i, X^{\otimes n}, X_j}$ is simply given by the stacking of the braid onto the $n$-strands labeled by $X$.

\begin{theorem}\label{MainThm}

\begin{enumerate}

\item Any $(d,3,1)$-gYBE object $X$ leads to a solution of the $(d,3,1)$-gYBE which satisfies far commutativity.

\item Let $L=\{Y_k\}$ be a representative set of all isomorphism classes of simple objects of $\mC$.  Then the number of distinct eigenvalues of each braid generator is less or equal to $l$, where $l=\sum_{Y_k\in L}\textrm{dim(Hom}(Y_k, X\otimes X))$.  Note that the resulting braid group representation is always a direct sum over $i,j\in I$ and therefore reducible.

\item The quantum dimension of a gYBE object is always an integer.  Actually, $\textrm{dim}(X)=d$, where $d$ is also the cardinality of the index set $I$.

\end{enumerate}

\end{theorem}

\begin{proof}

(1):  Using the graphical notation above, the $(d,3,1)$-gYBE solution is given by the braiding matrix $R=\oplus_{i,j\in I}M^{X_iXX}_{X_j}: \oplus_{i,j\in I} V_{X_i,X^{\otimes 2},X_j}\longrightarrow \oplus_{i,j\in I} V_{X_i,X^{\otimes 2},X_j}$, where $M^{X_iXX}_{X_j}$ consists of the braidings $R^{XX}_{Y_k}$ below in $(2)$ conjugated by the $F$-matrix $F^{X_iXX}_{X_j}$.  Because $R$ leads to a representation of $B_n$ for each $n$, it follows that $R$ satisfies far commutativity.

(2):  Using graphical calculus, we find that the eigenvalues of a $(d,3,1)$-$R$-matrix given in $(1)$ consist of the eigenvalues of the braidings $\{R^{XX}_{Y_k}, Y_k\in L\}$ for all $Y_k$ such that the fusion coefficients $N_{XX}^{Y_k}\neq 0$ and $N_{X_iX_j}^{Y_k}\neq 0$.  This follows from the fact that the representation matrix for a braid generator on $V_{X_i,X^{\otimes 2},X_j}$ is the conjugation of the braiding matrix by the  $F$-matrix $F^{X_iXX}_{X_j}$.

(3):  From $X\otimes X_i\cong \oplus_{j\in I}X_j$, we have $d_Xd_{X_i}=\sum_{j\in I} d_{X_j}$.  Thus, the quantum dimension $d_{X_i}$ is independent of $i$.  It follows that $d_Xd_{X_i}=d d_{X_i}$ and therefore that $d_X=d$.

\end{proof}

To provide examples of $(d,3,1)$-gYBE objects, we consider the unitary modular category $SO(N)_2$ described in Section $3.1$ of \cite{NR}.  When $N=2r+1, r\geq 1$, $SO(N)_2$ is of rank $r+4$ with simple objects denoted as $\{X_0, X_{2\lambda_1}, X_{\gamma^i}, 1\leq i \leq r, X_{\epsilon}, X_{\epsilon'}\}$ in \cite{NR}.  We will denote $X_0$ as $1$, $X_{2\lambda_1}$ as $Z$, and $X_{\gamma^i}$ as $X_i$ below.  All fusion rules follow from the following:

\begin{enumerate}

\item $X_{\epsilon}\otimes X_{\epsilon}\cong 1 \oplus \oplus_{i=1}^{i=r}X_{i}$

\item $X_{\epsilon}\otimes X_{i}\cong X_{\epsilon}\oplus X_{\epsilon'}, 1\leq i \leq r$

\item $X_{\epsilon}\otimes X_{\epsilon'}\cong  Z\oplus \oplus_{i=1}^{i=r}X_{i}$

\item $Z \otimes X_{\epsilon}=X_{\epsilon'}$

\item $Z \otimes Z=1$

\item $Z \otimes X_{i}=X_i, 1\leq i \leq r$

\item $X_i\otimes X_i=1\oplus Z\oplus X_{\textrm{min}\{2i, 2r+1-2i\}}, 1\leq i \leq r$

\item $X_i\otimes X_j=X_{j-i}\oplus X_{\textrm{min}\{i+j, 2r+1-i-j\}}, i<j$

\end{enumerate}

By examining the fusion rules, we observe

\begin{theorem}
The simple objects $X_{i}, 1\leq i \leq r,$ are $(2,3,1)$-gYBE objects with respect to the set $S=\{X_{\epsilon}, X_{\epsilon'}\}$.
\end{theorem}

\begin{remark}
In order to write down a $(d,3,1)$-$R$-matrix as an explicit matrix, we need conventions for ordering bases and tensoring matrices.  For tensoring matrices, we will use the Kronecker product, which means bases will be ordered lexicographically.
\end{remark}

 For an explicit example, we compute the gYBE solution given by the Jones-Kauffman theory at a $6^{th}$ root of unity, closely related to $SO(3)_2$.  Let the Kauffman variable $A=ie^{-\frac{\pi i}{12}}$, then the label set consists of $L=\{0,1,2,3,4\}$ \cite{Wang}.  In the notation above, $X_0=0,X_\epsilon=1, X_{\epsilon'}=3,X_1=2,Z=4$.  Thus, the object $2$ is a $(2,3,1)$-gYBE object with respect to simple objects $\{1,3\}$.  The $(2,3,1)$-gYBE-$R$-matrix is the braiding matrix for the braid generator $\sigma_2$ with respect to the following tree basis, where $i\in \{1,3\}$:

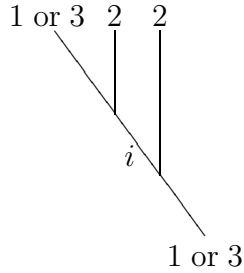
\[ \xy
(0,10)*{}="A"; (4,10)*{}="B";(8,10)*{}="C"; (14,10)*{}="K"; (14,-9.2)*{}="L"; (12,-10)*{}="D";(20,-17.3)*{}="H";
(4,3.8)*{}="E";(8,-1.2)*{}="F"; "A";"H" **\dir{-};
"C";"F" **\dir{-}; "K";"L" **\dir{-};(-1,12)*{1\;\textrm{or}\;3};
(8,12)*{2};(20,-20)*{1\;\textrm{or}\;3}; (14,12)*{2}; (10,-6.6)*{i};
\endxy \]

The internal edge can be labeled by either $1$ or $3$, just like the other two slant edges.  The qubit space ${\mathbb{C}}^2$ is spanned by the two labels $\{1,3\}$ with a basis denoted as $\{e_i,i=1,3\}$.  Therefore, the representation space for $\sigma_2$ has a basis $\{e_{ijk}=e_i\otimes e_j\otimes e_k \}$ for $i,j,k\in \{1,3\}$,  which can be identified as $({\mathbb{C}}^2)^{\otimes 3}$.  For matrix tensor products, we use the Kronecker product.  Then the basis is ordered lexicographically as
$$\{e_{111}, e_{113}, e_{131}, e_{133}, e_{311}, e_{313}, e_{331}, e_{333}\},$$
It follows that the $(2,3,1)$-$R$-matrix is the $8\times 8$ matrix given in the form $R=P^{-1}BP$, where $B=M^{122}_1 \oplus M^{322}_1 \oplus M^{122}_3 \oplus M^{322}_3$ is a braid matrix and $P$ is a permutation matrix. The natural order to compute the braid matrix $B$ is
$$\{e_{111}, e_{131}, e_{113}, e_{133}, e_{311}, e_{331}, e_{313}, e_{333}\}.$$
Furthermore, we need to order the bases for the braidings $\{R^{22}_{i}, i=0,2,4\}$.  Our orderings are $\{i=0, i=2\}$ and $\{i=4, i=2\}$.
Under this ordering of the bases, each $2\times 2$ matrix $M^{abc}_d$ is the conjugation of an $R$-matrix by the $F$-matrix $F^{abc}_d$, where the $R$ matrix is diagonal.  Explicitly as an $8\times 8$ matrix, the braid matrix $B$ is

$$
\begin{pmatrix}
H \begin{pmatrix}
e^{\frac{2\pi i}{3}} & 0\\
0 & e^{\frac{4\pi i}{3}}
\end{pmatrix} H
&0 &0 &0  \\
0&H \begin{pmatrix}
e^{\frac{5\pi i}{3}} & 0\\
0 & e^{\frac{4\pi i}{3}}
\end{pmatrix} H
&0 &0 \\
0&0&
H \begin{pmatrix}
e^{\frac{5\pi i}{3}} & 0\\
0 & e^{\frac{4\pi i}{3}}
\end{pmatrix} H
&0 \\
0& 0& 0&
H \begin{pmatrix}
e^{\frac{2\pi i}{3}} & 0\\
0 & e^{\frac{4\pi i}{3}}
\end{pmatrix} H
\end{pmatrix},$$ where $H$ is the Hadamard matrix
$H=\frac{1}{\sqrt{2}}\begin{pmatrix}
1 & 1\\
1& -1
\end{pmatrix}$.
The permutation matrix results from the order changing of the two bases, so it is the block sum $P=1\oplus \sigma_x \oplus 1 \oplus 1\oplus \sigma_x \oplus 1$, where $\sigma_x$ is the Pauli matrix
$\sigma_x=\begin{pmatrix}
0 & 1\\
1 & 0
\end{pmatrix}$.   Note that $P=P^{-1}$.
Computing everything, we find the $(2,3,1)$-$R$-matrix to be
$$
\begin{pmatrix}
-\frac{1}{2}& 0 & \frac{\sqrt{3}}{2} i &0 &0 &0 &0 &0 \\
0& -\frac{\sqrt{3}}{2} i &0 &\frac{1}{2} &0 &0 &0 &0 \\
\frac{\sqrt{3}}{2} i &0 &-\frac{1}{2} &0 &0 &0 &0 &0 \\
0& \frac{1}{2}& 0 & -\frac{\sqrt{3}}{2} i &0 &0 &0 &0 \\
0&0&0&0& -\frac{\sqrt{3}}{2} i &0 &\frac{1}{2} &0 \\
0&0&0&0&0&-\frac{1}{2}&0& \frac{\sqrt{3}}{2} i \\
0&0&0&0& \frac{1}{2} &0 &-\frac{\sqrt{3}}{2} i &0 \\
0&0&0&0& 0 & \frac{\sqrt{3}}{2} i &0 &-\frac{1}{2}
\end{pmatrix}.$$

Since the distinct eigenvalues of the $(2,3,1)$-$R$-matrix $R$ are $\{e^{\frac{2\pi i}{3}},e^{\frac{4\pi i}{3}}, e^{\frac{5\pi i}{3}}\}$, this solution is different from the four known solutions in \cite{RZWG}\cite{RC}.

The referee pointed out a $(3,3,1)$-gYBE simple object in $SU(3)_3$.  The unitary modular category $SU(3)_3$ is of rank $10$.  Using the notations in \cite{Rowell}, the simple object $Y$ of dimension $3$ is a $(3,3,1)$-gYBE object with respect to $\{X_1, {X_2}^{*}, X_3\}$ (see the Bratelli diagram in Fig. $1$ of \cite{Rowell} for the fusion rules).

Finally, we make the following conjecture.

\begin{conj}
Images of resulting braid group representations from $(d,3,1)$-gYBE objects in unitary ribbon fusion categories are always finite groups.
\end{conj}


\begin{thebibliography}{AA}

\bibitem[C]{RC}R.\ S. Chen, \emph{Generalized Yang-Baxter equations and Braiding Quantum Gates}, Journal of Knot Theory and its Ramifications (to appear), Arxiv: arXiv:1108.5215.

\bibitem[NR]{NR}D.\ Naidu, and E.\ C.\ Rowell, \emph{A finiteness property for braided fusion categories}, Algebr. and Represent. Theor. 14 (2011), no. 5, 837-855, arXiv:0903.4157

\bibitem[R]{Rowell}E.\ C. Rowell, \emph{A quaternionic braid representation (after Goldschmidt and Jones)}, Quantum Topol.  2 (2011), 173-182, arXiv:1006.4808.

\bibitem[RZWG]{RZWG}E.\ C.\ Rowell, Y.\ Zhang, Y.~S.\ Wu, and M.~L.\ Ge, \emph{Extraspecial Two-Groups, Generalized Yang-Baxter Equations and Braiding Quantum Gates,} Quantum Inf.\ Comput.\ 10 (2010) no. 7-8, 685–--702, arXiv:0706.1761.

\bibitem[W]{Wang}Z.\ Wang, \emph{Topological quantum computation}, CBMS monograph vol. 112, American Mathematical Society, 2010.


\end{thebibliography}
\end{document}